\documentclass[reqno,12pt]{amsart}
\usepackage{amsfonts,amsthm,amsmath,amssymb,bm,mathrsfs,comment}
\usepackage[dvipdfmx]{color,graphicx}

\newtheorem{theorem}{Theorem}

\newtheorem{prop}{Proposition}

\theoremstyle{definition}

\theoremstyle{remark}

\numberwithin{equation}{section}

\usepackage{amsmath,amssymb}
\usepackage{txfonts}
\usepackage{fancybox} 
\usepackage{ulem}
\usepackage{wrapfig}
\usepackage{comment}
\usepackage{amsthm,color,bm}

\def\re{\mathbb{R}}

\def\Z{\mathbb{Z}}

\def\({\left(}
\def\){\right)}
\def\pd{\partial}

\def\la{\lambda}
\def\al{\alpha}

\begin{document}
\title[]{Nondegeneracy of the entire solution for the $N$-Laplace Liouville equation}

\author{Futoshi Takahashi}
\address{Department of Mathematics, Osaka Metropolitan University \& OCAMI, Sumiyoshi-ku, Osaka, 558-8585, Japan}
\email{futoshi@omu.ac.jp}


\date{\today}

\dedicatory{}

\begin{abstract}
In this note, we prove the nondegeneracy of the explicit finite-mass solution to the $N$-Laplace Liouville equation on the whole space, 
which is recently shown to be unique up to scaling and translation.
\end{abstract}

\maketitle

%
%
\section{Introduction}

Let $N \ge 2$ be an integer.
In this note, we concern the following quasi-linear Liouville equation:
\begin{align}
\label{Liouville}
	\begin{cases}
	&-\Delta_N U = e^U \quad \text{in} \quad \re^N, \\
	&\int_{\re^N} e^U dx < \infty,
	\end{cases}
\end{align}
here $\Delta_N U = {\rm div } (|\nabla U|^{N-2} \nabla U)$ denotes the $N$-Laplacian of a function $U$.
Problem \eqref{Liouville} has the explicit solution (Liouville bubble)
\begin{equation}
\label{bubble}
	U(x) = \log \frac{C_N}{\( 1 + |x|^{\frac{N}{N-1}} \)^N}, \quad x \in \re^N,
\end{equation}
where $C_N = N \(\frac{N^2}{N-1}\)^{N-1}$.
Thanks to the scaling and translation invariance of the problem, the functions
\begin{equation}
\label{scale}
	U_{\la, a}(x) = U(\la(x-a)) + N \log \la, \quad \la > 0, a \in \re^N
\end{equation}
constitute a $(N+1)$-dimensional family of solutions to \eqref{Liouville} with
\[
	\int_{\re^N} e^{U_{\la,a}} dx = \(\frac{\omega_{N-1}}{N}\) C_N,
\]
where $\omega_{N-1}$ denotes the area of the unit sphere $S^{N-1}$ in $\re^N$.
Indeed, all the solutions of \eqref{Liouville} are of the form \eqref{scale}.
This fact is first proven by Chen and Li \cite{Chen-Li} when $N = 2$ by the method of moving planes.
Recently, P. Esposito \cite{Esposito} proves the same classification result for \eqref{Liouville} when $N \ge 3$.
His method exploits a weighted Sobolev estimates at infinity for any solution to \eqref{Liouville}, an isoperimetric argument, and the Pohozaev identity,
and does not use the moving plane arguments. 

We are interested in the linear nondegeneracy of the explicit solution $U$ in \eqref{bubble}.
Thus we consider the linearized operator around $U$:
\[
	L(\phi) = \frac{d}{dt} \Big |_{t=0} N(U + t \phi)
\]
where $N(U + t\phi) = \Delta_N (U + t\phi) + e^{U + t\phi}$ for $t \in \re$ and a function $\phi$.
Then we compute directly that
\begin{equation}
\label{L}
	L(\phi) = {\rm div } (|\nabla U|^{N-2} \nabla \phi) + (N-2) {\rm div } (|\nabla U|^{N-4} (\nabla U \cdot \nabla \phi) \nabla U) + e^U \phi,
\end{equation}
here and henceforth, $``\cdot"$ denotes the standard inner product in $\re^N$.
Since $U_{\la,a}$ in \eqref{scale} solves the equation
\[
	\Delta_N U_{\la.a} + e^{U_{\la,a}} = 0 \quad \text{in} \quad \re^N,
\]
by differentiating the above equation with respect to the parameters $\la$ and $a_1, \dots, a_N$ at $\la = 1$ and $a = 0$,
we obtain the bounded solutions 
\begin{align}
\label{Z_0}
	&Z_0(x) = \frac{d}{d\la} \Big|_{\la=1, a=0} U_{\la,a} = x \cdot \nabla U + N, \\
\label{Z_i}
	&Z_i(x) = \frac{d}{d a_i} \Big|_{\la=1, a=0} U_{\la,a} = \frac{\pd U}{\pd x_i}, \quad (i=1,\dots, N)
\end{align}
to the linearized equation $L(\phi) = 0$.

The aim of this note is to prove the following nondegeneracy of $U$:

\begin{theorem}
\label{Theorem:nondegeneracy}
Let $U$ be as in \eqref{bubble} and let $\phi$ be a solution in $L^{\infty} \cap C^2(\re^N)$ to the linearized equation $L(\phi) = 0$,
here $L$ is as in \eqref{L}.
Then $\phi$ can be written as a linear combination of $Z_0, Z_1, \dots, Z_N$ defined by \eqref{Z_0}, \eqref{Z_i}. 
\end{theorem}

The above theorem was known already when $N = 2$, see \cite{Baraket-Pacard}, \cite{El Mehdi-Grossi}.
In this note, we extend the result to $N \ge 3$.

Our proof is similar to that of \cite{Pistoia-Vaira}, in which the authors study the critical $p$-Laplace equation 
\[
	-\Delta_p U = U^{p^*-1} \quad \text{in} \quad \re^N, \quad U > 0,
\]
where $p^* = \frac{Np}{N-p}$, $1 < p < N$, is the critical Sobolev exponent.
They prove the linear nondegeneracy of the explicit entire solution (Aubin-Talenti bubble)
\[
	U(x) = \( \frac{\al_{N,p}}{1 + |x|^{\frac{p}{p-1}}}\)^{\frac{N-p}{p}}
\]
where $\al_{p,N} = N^{\frac{1}{p}} \( \frac{N-p}{p-1} \)^{\frac{p-1}{p}}$,
extending the former result by Rey \cite{Rey} for $p = 2$.

%
%
\section{Proof of Theorem \ref{Theorem:nondegeneracy}}

In this section, we prove Theorem \ref{Theorem:nondegeneracy}.
We follow the method by \cite{Pistoia-Vaira}. See also \cite{Baraket-Pacard}, \cite{El Mehdi-Grossi}.

First, we prove the next proposition:
\begin{prop}
Let $L$ is as in \eqref{L}. Then $\phi \in C^2(\re^N)$ solves $L(\phi) = 0$ if and only if $\phi$ is a solution to
\begin{align}
\label{PV(3.1)}
	&|x|^2 \Delta \phi + N(N-2) \frac{(x \cdot \nabla \phi)}{1 + |x|^{\frac{N}{N-1}}} + (N-2) \sum_{i,j=1}^N \frac{\pd^2 \phi}{\pd x_i \pd x_j} x_i x_j \\
	&+ \( \frac{N^3}{N-1} \) \frac{|x|^{\frac{N}{N-1}}}{\( 1 + |x|^{\frac{N}{N-1}} \)^2} \phi = 0. \notag
\end{align}
\end{prop}

\begin{proof}
We rewrite the equation $L(\phi) = 0$ as
\begin{align*}
	L(\phi) &= {\rm div } (|\nabla U|^{N-2} \nabla \phi) + (N-2) {\rm div } (|\nabla U|^{N-4} (\nabla U \cdot \nabla \phi) \nabla U) + e^U \phi \\
	&= |\nabla U|^{N-2} \Delta \phi \\
	&+ \nabla (|\nabla U|^{N-2}) \cdot \nabla \phi \\
	&+ (N-2) |\nabla U|^{N-4} (\nabla U \cdot \nabla \phi) \Delta U \\
	&+ (N-2) (\nabla U \cdot \nabla \phi) \nabla (|\nabla U|^{N-4}) \cdot \nabla U \\
	&+ (N-2) |\nabla U|^{N-4} \nabla ( \frac{1}{2} |\nabla U|^2 ) \cdot \nabla \phi \\
	&+ (N-2) |\nabla U|^{N-4} (D^2\phi) (\nabla U, \nabla U) \\
	&+ e^U \phi \\
	&=  A + B + C + D + E + F + G = 0,
\end{align*}
where we have used 
\[
	\nabla (\nabla U \cdot \nabla \phi) \cdot \nabla \phi =  \nabla ( \frac{1}{2} |\nabla U|^2 ) \cdot \nabla \phi + D^2\phi (\nabla U, \nabla U)
\]
with the notation that $(D^2\phi) (\nabla U, \nabla U) = \sum_{i,j=1}^N \frac{\pd^2 \phi}{\pd x_i \pd x_j} \frac{\pd U}{\pd x_i} \frac{\pd U}{\pd x_j}$.

Now, we compute that
\begin{align*}
	&\nabla U = - \(\frac{N^2}{N-1}\) \frac{|x|^{\frac{1}{N-1}}}{1 + |x|^{\frac{N}{N-1}}} \frac{x}{|x|}, \\
	&|\nabla U|^k = \(\frac{N^2}{N-1}\)^k \frac{|x|^{\frac{k}{N-1}}}{(1 + |x|^{\frac{N}{N-1}})^k}, \quad (k \in \Z) \\
	&\nabla (|\nabla U|^k) =  \(\frac{N^2}{N-1}\)^k \(\frac{k}{N-1} \) \frac{|x|^{\frac{k}{N-1}-1}}{\( 1 + |x|^{\frac{N}{N-1}} \)^{k+1}}  \left\{ 1+ (1 - N) |x|^{\frac{N}{N-1}} \right\} \frac{x}{|x|}, \quad (k \in \Z).
\end{align*}
Thus we have
\begin{align*}
	&\nabla U \cdot \nabla \phi = -\(\frac{N^2}{N-1}\) \frac{|x|^{\frac{1}{N-1}-1}}{1 + |x|^{\frac{N}{N-1}}} (x \cdot \nabla \phi), \\
	&\nabla \( |\nabla U|^{N-4} \) \cdot \nabla U = - \(\frac{N^2}{N-1}\)^{N-3} \(\frac{N-4}{N-1}\) \frac{|x|^{\frac{-2}{N-1}}}{\( 1 + |x|^{\frac{N}{N-1}} \)^{N-2}} \left\{ 1+ (1 - N) |x|^{\frac{N}{N-1}} \right\}, \\
	&(D^2 \phi)(\nabla U, \nabla U) = \(\frac{N^2}{N-1}\)^2 \frac{|x|^{\frac{2}{N-1}-2}}{( 1+ |x|^{\frac{N}{N-1}} )^2} \sum_{i, j = 1}^N \frac{\pd^2 \phi}{\pd x_i \pd x_j} x_i x_j.
\end{align*}
Also we see
\begin{align*}
	&\Delta U = -\( \frac{N^2}{N-1} \) \frac{|x|^{\frac{1}{N-1}-1}}{\(1 + |x|^{\frac{N}{N-1}}\)^2} \left\{ \( N-1 + \frac{1}{N-1} \) + (N-2) |x|^{\frac{N}{N-1}} \right\}.
\end{align*}
From these, we obtain
\begin{align*}
	A &= |\nabla U|^{N-2} \Delta \phi =  \(\frac{N^2}{N-1}\)^{N-2} \frac{|x|^{\frac{N-2}{N-1}}}{(1 + |x|^{\frac{N}{N-1}})^{N-2}} \Delta \phi, \\
	B &= \nabla (|\nabla U|^{N-2}) \cdot \nabla \phi \\
	&= \(\frac{N^2}{N-1}\)^{N-2} \(\frac{N-2}{N-1}\) \frac{|x|^{\frac{-N}{N-1}}}{(1 + |x|^{\frac{N}{N-1}})^{N-1}} \left\{ 1+ (1 - N) |x|^{\frac{N}{N-1}} \right\} (x \cdot \nabla \phi), \\
	C &=(N-2) |\nabla U|^{N-4} (\nabla U \cdot \nabla \phi) \Delta U \\
	&= (N-2) \(\frac{N^2}{N-1}\)^{N-2} \frac{|x|^{\frac{-N}{N-1}}}{(1 + |x|^{\frac{N}{N-1}})^{N-1}} \left\{ \( N-1 + \frac{1}{N-1} \) + (N-2)|x|^{\frac{N}{N-1}} \right\} (x \cdot \nabla \phi), \\
	D &= (N-2)(\nabla U \cdot \nabla \phi) \nabla (|\nabla U|^{N-4}) \cdot \nabla U \\
	&= (N-2) \(\frac{N^2}{N-1}\)^{N-2} \(\frac{N-4}{N-1}\) \frac{|x|^{\frac{-N}{N-1}}}{(1 + |x|^{\frac{N}{N-1}})^{N-1}} \left\{ 1 + (1-N) |x|^{\frac{N}{N-1}} \right\} (x \cdot \nabla \phi), \\
	E &= (N-2) |\nabla U|^{N-4} \nabla \( \frac{1}{2} |\nabla U|^2 \) \cdot \nabla \phi \\
	&= \(\frac{N^2}{N-1}\)^{N-2} \(\frac{N-2}{N-1}\) \frac{|x|^{\frac{-N}{N-1}}}{(1 + |x|^{\frac{N}{N-1}})^{N-1}} \left\{ 1 + (1-N) |x|^{\frac{N}{N-1}} \right\} (x \cdot \nabla \phi), \\
	F &= (N-2) |\nabla U|^{N-4} (D^2\phi) (\nabla U, \nabla U) \\
	&= (N -2)\(\frac{N^2}{N-1}\)^{N-2} \frac{|x|^{\frac{-N}{N-1}}}{(1 + |x|^{\frac{N}{N-1}})^{N-2}}  \sum_{i, j = 1}^N \frac{\pd^2 \phi}{\pd x_i \pd x_j} x_i x_j, \\
	G &= e^U \phi = \frac{C_N}{(1 + |x|^{\frac{N}{N-1}})^N} \phi.
\end{align*}
Returning to the equation $L(\phi) = 0$ with these expressions and after some manipulations, 
we obtain that $L(\phi) = 0$ is equivalent to that $\phi$ satisfies \eqref{PV(3.1)}.
\end{proof}

\vspace{1em}\noindent
{\it Proof of Theorem \ref{Theorem:nondegeneracy}.}

As in \cite{Baraket-Pacard}, \cite{El Mehdi-Grossi}, and \cite{Pistoia-Vaira}, we decompose a solution $\phi$ to \eqref{PV(3.1)} by using spherical harmonics.
Let us denote $x = r \omega$, $r = |x|$, $\omega = \frac{x}{|x|} \in S^{N-1}$ for a point $x \in \re^N$.
We write
\begin{equation}
\label{decomposition}
	\phi(x) = \phi(r, \omega) = \sum_{k=0}^{\infty} \psi_k(r) Y_k(\omega), \quad \psi_k(r) = \int_{S^{N-1}} \phi(r, \omega) Y_k(\omega) dS_{\omega},
\end{equation}
where $Y_k(\omega)$ denote the $k$-th spherical harmonics, 
that is, the $k$-th eigenfunctions for the Laplace-Beltrami operator $\Delta_{S^{N-1}}$ on $S^{N-1}$ associated with the $k$-th eigenvalue $\la_k$:
\[
	-\Delta_{S^{N-1}} Y_k = \la_k Y_k \quad \text{on} \, S^{N-1}
\]
where
\[
	\la_k = k(k + N -2), \quad k=0,1,2,\dots,
\]
denotes the $k$-th eigenvalue.
It is known that the multiplicity of $\la_k$ is $\frac{(2k + N-2)(N + k-3)!}{k! (N-2)!}$,
especially, $\la_0 = 0$ has the multiplicity $1$ and $\la_1 = N-1$ has the multiplicity $N$.

We derive the equation satisfied by $\psi_k$ for $k=0,1,2, \dots$.
Let $\nabla_{\omega}$ denote the spherical gradient operator on $S^{N-1}$.
Since the decomposition of the gradient operator
\[
	\nabla = \omega \frac{\pd}{\pd r} + \frac{1}{r} \nabla_{\omega}, \quad \omega \cdot \nabla_{\omega} \equiv 0
\]
holds, 
for a function $\phi$ of the form $\phi(x) = \psi(r) Y(\omega)$,
we have
\begin{align*}
	&x \cdot \nabla \phi = x \cdot \nabla (\psi(r) Y(\omega)) = r \psi'(r) Y(\omega), \\ 
	&\sum_{i,j=1}^N \frac{\pd^2 \phi}{\pd x_i \pd x_j} x_i x_j = \sum_{i,j=1}^N \frac{\pd^2 (\psi(r) Y(\omega))}{\pd x_i \pd x_j} x_i x_j = r^2 \psi''(r) Y(\omega).
\end{align*}
Also recall the formula
\[
	\Delta = \frac{\pd^2}{\pd r^2} + \frac{N-1}{r} \frac{\pd}{\pd r} + \frac{1}{r^2} \Delta_{S^{N-1}}.
\]
Thus we have, for $\phi$ of the form $\phi(x) = \psi(r) Y(\omega)$, the equation \eqref{PV(3.1)} becomes
\begin{align*}
	&r^2 \( \psi''(r) + \frac{N-1}{r} \psi'(r) \) Y(\omega) + \psi(r) \Delta_{S^{N-1}} Y(\omega) \\
	&+ N(N-2) \frac{r \psi'(r) Y(\omega)}{1 + r^{\frac{N}{N-1}}} + (N-2) r^2 \psi''(r) Y(\omega) 
	+ \( \frac{N^3}{N-1} \) \frac{r^{\frac{N}{N-1}}}{\( 1 + r^{\frac{N}{N-1}} \)^2} \psi(r) Y(\omega) = 0.
\end{align*}
Thus inserting \eqref{decomposition} into \eqref{PV(3.1)}, we see that each $\psi_k$ must be a solution to
\begin{align}
\label{L_k=0}
	L_k(\psi) & :=\psi''(r) + \( 1 + \frac{N(N-2)}{N-1} \frac{1}{1 + r^{\frac{N}{N-1}}} \)  \frac{\psi'(r)}{r} \\ 
	&-\frac{\la_k}{N-1} \frac{\psi(r)}{r^2} + \frac{N^3}{(N-1)^2} \frac{r^{\frac{N}{N-1}}}{\( 1 + r^{\frac{N}{N-1}} \)^2} \frac{\psi(r)}{r^2} = 0. \notag
\end{align}
Also note that, by using the expression $U(r) = \log \frac{C_N}{\( 1 + r^{\frac{N}{N-1}} \)^N}$, we see that the equation 
\[
	L(\phi) = {\rm div } (|\nabla U|^{N-2} \nabla \phi) + (N-2) {\rm div } (|\nabla U|^{N-4} (\nabla U \cdot \nabla \phi) \nabla U) + e^U \phi = 0
\]
for $\phi(x) = \psi(r) Y(\omega)$ is equivalent to that $\psi$ satisfies
\begin{align}
\label{L_k=0bis}
	\left\{ r^{N-1} \psi'(r) |U'(r)|^{N-2} \right\}' - \la_k r^{N-3} \(\frac{1}{N-1}\) |U'(r)|^{N-2} \psi(r) + \frac{e^{U(r)}}{N-1} r^{N-1} \psi(r) = 0.
\end{align}
In the following, we treat the equation $L_k(\psi) = 0$ in \eqref{L_k=0} for $k=0$, $k=1$, and $k \ge 2$ separately.

\vspace{1em}\noindent
{\bf The case $k=0$.} 

By the invariance under the scaling, we know that $Z_0(x)$ defined in \eqref{Z_0} satisfies \eqref{PV(3.1)}.
Since
\[
	Z_0(x) = x \cdot \nabla U(x) + N = \(\frac{N}{N-1}\) \underbrace{\frac{(N-1) - |x|^{\frac{N}{N-1}}}{1 + |x|^{\frac{N}{N-1}}}}_{:= \psi_0(|x|)},
\]
we see that
\[
	\psi_0(r) = \frac{(N-1) - r^{\frac{N}{N-1}}}{1 + r^{\frac{N}{N-1}}}
\]
is a solution of $L_0(\psi) = 0$, which is bounded on $[0, +\infty)$. 

We claim that any other bounded solution of $L_0(\psi) = 0$ must be a constant multiple of $\psi_0$.
Indeed, assume the contrary that there existed the second linearly independent, bounded solution $\psi$ satisfying $L_0(\psi) = 0$.
We may always assume that $\psi$ is of the form
\[
	\psi(r) = c(r) \psi_0(r)
\]
for some $c = c(r)$.
Inserting this into \eqref{L_k=0} and noting $\la_0 = 0$, we obtain
\begin{align*}
	&c''(r) \psi_0(r) + c'(r) \left[ 2 \psi_0'(r) + \frac{\psi_0(r)}{r} \( 1 + \frac{N(N-2)}{N-1} \frac{1}{1 + r^{\frac{N}{N-1}}} \) \right] \\
	& + c \underbrace{\left[ \psi_0''(r) + \( 1 + \frac{N(N-2)}{N-1} \frac{1}{1 + r^{\frac{N}{N-1}}} \)  \frac{\psi_0'(r)}{r}
	+ \frac{N^3}{(N-1)^2} \frac{r^{\frac{N}{N-1}}}{\( 1 + r^{\frac{N}{N-1}} \)^2} \frac{\psi_0(r)}{r^2} \right]}_{=0} = 0,
\end{align*}
which leads to
\begin{align*}
	\frac{c''(r)}{c'(r)} = -2 \frac{\psi_0'(r)}{\psi_0(r)} - \frac{1}{r} \( 1 + \frac{N(N-2)}{N-1} \frac{1}{1 + r^{\frac{N}{N-1}}} \). 
\end{align*}
This can be written as
\begin{align*}
	(\log |c'(r)|)' = -2 (\log |\psi_0(r)|)' - \( 1 + \frac{N(N-2)}{N-1} \) (\log r)' + (N-2) \left\{ \log \(1 + r^{\frac{N}{N-1}} \) \right\}', 
\end{align*}
so we have that
\[
	c'(r) = A \frac{\( 1 + r^{\frac{N}{N-1}} \)^{N-2}}{\psi_0^2(r) r^{1 + \frac{N(N-2)}{N-1}}} 
\]
for some $A \ne 0$.
Since $\psi_0(r) \sim -1$ near $r=\infty$, 
we have
\[
	c'(r) \sim A \frac{r^{\frac{N(N-2)}{N-1}}}{r^{1 + \frac{N(N-2)}{N-1}}} = \frac{A}{r} \quad \text{as} \, r \to \infty
\]
which implies $c(r) \sim A \log r + B$ as $r \to \infty$ for some $A \ne 0$ and $B \in \re$.
However, in this case, $|\psi(r)| \sim |(A \log r + B) \psi_0(r)| \to +\infty$ as $r \to +\infty$, which contradicts to the assumption that $\psi$ is bounded.
Therefore, we obtain the claim.

\vspace{1em}\noindent
{\bf The case $k=1$.} 

By the invariance under the translation, we know that $Z_i(x)$ $(i=1,\dots, N)$ defined in \eqref{Z_i} satisfies \eqref{PV(3.1)}.
Since
\[
	Z_i(x) = \frac{\pd U}{\pd x_i} = - \(\frac{N^2}{N-1}\) \underbrace{\frac{r^{\frac{1}{N-1}}}{1 + r^{\frac{N}{N-1}}}}_{= \psi_1(r)} \frac{x_i}{|x|}, \quad (i=1,\dots, N),
\]
we see that 
\[
	\psi_1(r) = \frac{r^{\frac{1}{N-1}}}{1 + r^{\frac{N}{N-1}}}
\]
is a solution of $L_1(\psi) = 0$, which is bounded (decaying) on $[0, +\infty)$: $\psi_1(r) \sim \frac{1}{r}$ as $r \to +\infty$. 

As before, we claim that any other bounded solution of $L_1(\psi) = 0$ must be a constant multiple of $\psi_1$.
Indeed, assume the contrary that there existed the second linearly independent, bounded solution $\psi$ satisfying $L_1(\psi) = 0$.
We may always assume that $\psi$ is of the form
\[
	\psi(r) = c(r) \psi_1(r)
\]
for some $c = c(r)$.
Inserting this into \eqref{L_k=0} and noting $\la_1 = N-1$, we obtain
\begin{align*}
	&c''(r) \psi_1(r) + c'(r) \left[ 2 \psi_1'(r) + \frac{\psi_1(r)}{r} \( 1 + \frac{N(N-2)}{N-1} \frac{1}{1 + r^{\frac{N}{N-1}}} \) \right] \\
	& + c \underbrace{\left[ \psi_1''(r) + \( 1 + \frac{N(N-2)}{N-1} \frac{1}{1 + r^{\frac{N}{N-1}}} \) \frac{\psi_1'(r)}{r} -\frac{\psi_1(r)}{r^2}
	+ \frac{N^3}{(N-1)^2} \frac{r^{\frac{N}{N-1}}}{\( 1 + r^{\frac{N}{N-1}} \)^2} \frac{\psi_1(r)}{r^2} \right]}_{=0} = 0,
\end{align*}
which leads to
\begin{align*}
	\frac{c''(r)}{c'(r)} = -2 \frac{\psi_1'(r)}{\psi_1(r)} - \frac{1}{r} \( 1 + \frac{N(N-2)}{N-1} \frac{1}{1 + r^{\frac{N}{N-1}}} \). 
\end{align*}
Again we have that
\[
	c'(r) = A \frac{\( 1 + r^{\frac{N}{N-1}} \)^{N-2}}{\psi_1^2(r) r^{1 + \frac{N(N-2)}{N-1}}} 
\]
for some $A \ne 0$.
Since $\psi_1(r) \sim \frac{1}{r}$ as $r \to +\infty$, we obtain
\[
	c'(r) \sim A \frac{r^{\frac{N(N-2)}{N-1}}}{r^{-1 + \frac{N(N-2)}{N-1}}} = A r \quad \text{as} \, r \to \infty
\]
which implies $c(r) \sim \frac{A}{2} r^2 + B$ as $r \to \infty$ for some $A \ne 0$ and $B \in \re$.
However, in this case, $\psi(r) \sim (\frac{A}{2} r^2 + B) \psi_1(r) \sim \frac{A}{2} r$ as $r \to +\infty$, which contradicts to the assumption that $\psi$ is bounded.
Therefore, we obtain the claim.

\vspace{1em}\noindent
{\bf The case $k \ge 2$.} 

In this case, we claim that all the bounded solutions of $L_k(\psi) = 0$ are identically zero. 
Assume the contrary that there existed $\psi \not\equiv 0$ satisfying $L_k(\psi) = 0$.
We may assume that there exists $R_k > 0$ such that $\psi(r) > 0$ on $(0, R_k)$ and $\psi'(R_k) \le 0$.
Now, $\psi$ satisfies \eqref{L_k=0bis}:
\begin{align}
\label{E_k}
	\left\{ r^{N-1} \psi'(r) |U'(r)|^{N-2} \right\}' - \la_k r^{N-3} \(\frac{1}{N-1}\) |U'(r)|^{N-2} \psi(r) + \frac{e^{U(r)}}{N-1} r^{N-1} \psi(r) = 0.
\end{align}
Also $\psi_1$ is a solution of \eqref{L_k=0bis} for $k=1$:
\begin{align}
\label{E_1}
	\left\{ r^{N-1} \psi_1'(r) |U'(r)|^{N-2} \right\}' - \la_1 r^{N-3} \(\frac{1}{N-1}\) |U'(r)|^{N-2} \psi_1(r) + \frac{e^{U(r)}}{N-1} r^{N-1} \psi_1(r) = 0.
\end{align}
Multiply \eqref{E_k} by $\psi_1$ and multiply \eqref{E_1} by $\psi_k$ and subtracting, we have
\[
	\left\{ r^{N-1} \psi_k'(r) |U'(r)|^{N-2} \right\}' \psi_1 - \left\{ r^{N-1} \psi_1'(r) |U'(r)|^{N-2} \right\}' \psi_k = \frac{\la_k - \la_1}{N-1} r^{N-3} |U'(r)|^{N-2} \psi_k \psi_1.
\]
Integrating both sides of the above from $r = 0$ to $r = R_k$ and using $\psi_k(R_k) = 0$, we obtain
\begin{align}
\label{contradiction}
	R_k^{N-1}|U'(r)|^{N-2} \psi_k'(R_k) \psi_1(R_k) = \frac{\la_k - \la_1}{N-1} \int_0^{R_k} r^{N-3}|U'(r)|^{N-2} \psi_k(r) \psi_1(r) dr.
\end{align}
Since $\la_k > \la_1$ for $k \ge 2$, $\psi_k(r) > 0$ on $(0, R_k)$, and $\psi_1(r) > 0$, the right-hand side of \eqref{contradiction} is positive.
On the other hand, the left-hand side of \eqref{contradiction} is non positive since $\psi_k'(R_k) \le 0$.
This contradiction implies the claim.

Combining all these facts, we have finished the proof of Theorem \ref{Theorem:nondegeneracy}.
\qed

%
%

\vspace{1em}\noindent
{\bf Acknowledgments.}

Part of this work was supported by 
JSPS Grant-in-Aid for Scientific Research (B), No.19H01800.
This work was partly supported by Osaka Central Advanced Mathematical Institute: MEXT Joint Usage/Research Center on Mathematics and Theoretical Physics JPMXP0619217849. 


\end{document}